\documentclass[12pt]{article}
\usepackage{amssymb,amsmath,amsthm, amsfonts}
\usepackage{graphicx}
\usepackage{epsfig}
\usepackage{tikz}

\def\disp{\displaystyle}

\def\dref#1{(\ref{#1})}

\theoremstyle{plain}
\newtheorem{theorem}{Theorem}[section]
\newtheorem{lemma}{Lemma}[section]

\theoremstyle{definition}

\newtheorem{remark}{Remark}[section]

\numberwithin{equation}{section}

\begin{document}

\title{\bf A note for global existence of a two-dimensional chemotaxis-haptotaxis model with remodeling of non-diffusible attractant}

\author{
Jiashan Zheng 
\thanks{Corresponding author.   E-mail address:
 zhengjiashan2008@163.com (J.Zheng)}
\\
    School of Mathematics and Statistics Science,\\
     Ludong University, Yantai 264025,  P.R.China \\
}
\date{}

\maketitle \vspace{0.3cm}
\noindent
\begin{abstract}
In this paper, we study the following  the coupled   chemotaxis--haptotaxis model with remodeling of non-diffusible attractant
$$
 \left\{\begin{array}{ll}
  u_t=\Delta u-\chi\nabla\cdot(u\nabla v)-
  \xi\nabla\cdot(u\nabla w)+\mu u(1-  u-w),\\
 \disp{v_t=\Delta v- v +u},\quad
\\
\disp{w_t=- vw+\eta w(1-u-w)},\quad\\
 \end{array}\right.\eqno(0.1)
$$
in a bounded smooth domain $\Omega\subseteq\mathbb{R}^2$ with zero-flux boundary conditions, where $\chi$, $\xi$  and $\eta$ are positive parameters.
Under appropriate regularity assumptions on the initial data $(u_0, v_0, w_0)$, by developing  some $L^p$-estimate techniques, we prove the global existence and uniqueness of classical solutions when $\mu>0$, where $\mu$ is the logistic growth rate of cancer cells. This result removes the additional restriction of $\mu$ {\bf is sufficiently large}
in \cite{PangPang1} (J. Diff. Eqns., 263(2)(2017),  1269--1292) for the global existence  of solutions.

\end{abstract}

\vspace{0.3cm}


\newpage
\section{Introduction}
The oriented movement of biological cells or organisms in response to a chemical
gradient is called chemotaxis (see
Calvez and Carrillo \cite{Calvez455},
Fontelos et al. \cite{Fontelos444},
Hillen and Painter \cite{Hillen79},
Horstmann \cite{Horstmann4455,Horstmann2710},
J\"{a}ger and Luckhaus \cite{Jger317},
Kavallaris and P. Souplet  \cite{Kavallaris333},
Nagai \cite{Nagai318},
Perthame \cite{Perthame444},
Sherratt \cite{Sherratt455},
Winkler \cite{Winkler792}).
To describe chemotaxis of cell populations,  the signal is  produced
by the cells, in 1970, Keller and Segel (see \cite{Keller79}) proposed
  an important variant of the quasilinear
chemotaxis model
%
\begin{equation}
 \left\{\begin{array}{ll}
  u_t=\nabla\cdot(\phi(u)\nabla u)-\chi\nabla\cdot(u\nabla v),\quad
x\in \Omega, t>0,\\
 \disp{v_t=\Delta v +u- v}.\quad
x\in \Omega, t>0.\\
 \end{array}\right.\label{ssdesdddrrff1.1}
\end{equation}
The interesting
feature of quasilinear Keller--Segel types of models \dref{ssdesdddrrff1.1}
 is the possibility of blow-up of solutions in finite time, which strongly depends
on the space dimension (see e.g. Horstmann et al. \cite{Horstmann4455,Horstmann444},
 Rascle and  Ziti \cite{Rascle444}).
%
In fact,
  solutions of \dref{ssdesdddrrff1.1} may blow up in finite time when
$N\geq2$ (Herrero and  Vel\'{a}zquez \cite{Herrero710},
Osaki et al. \cite{Osakix391}, Winkler \cite{Winkler793}).
In the
higher-dimensional case when $N\geq3$, small total mass of cells appears to be insufficient to
rule out blow-up in  (Winkler et al. \cite{Winkler792,Horstmann791}).
Some recent studies show that the large nonlinear diffusion function
(see Ishida  et al. \cite{Ishida},
Tao and Winkler \cite{Tao794}, Zheng
\cite{Zhengsssdddssddxxss,Zhengssddxxss}), the nonlinear chemotactic sensitivity function (see Fujie et al. \cite{Fujierrfttg}) and the (generalized)  logistic growth term (see Lankeit \cite{Lankeit3444dd79477ddffvg},
Winkler et al. \cite{Tello710,Winkler37103,Winkler79312}, Zheng \cite{Zheng,Zhengssddxxss}) may prevent the
blow-up of solutions. One important extension of the classical Keller--Segel model to a more complex cell migration
mechanism was proposed by Chaplain and Lolas (see Chaplain and Lolas \cite{Chaplain3444}) in order to describe processes of cancer
invasion. In 2006, Chaplain and Lolas (\cite{Chaplain3444}) described the process of cancer invasion on the macroscopic
scale by the chemotaxis-haptotaxis system (with remodeling of non-diffusible attractant)
 \begin{equation}
 \left\{\begin{array}{ll}
  u_t=\Delta u-\chi\nabla\cdot(u\nabla v)-\xi\nabla\cdot
  (u\nabla w)+\mu u(1-u-w),\quad
x\in \Omega, t>0,\\
 \disp{\tau v_t=\Delta v +u- v},\quad
x\in \Omega, t>0,\\
\disp{w_t=- vw +\eta w(1-u-w)},\quad
x\in \Omega, t>0,\\
 \disp{\frac{\partial u}{\partial \nu}-\chi u\frac{\partial v}{\partial \nu}-\xi\frac{\partial w}{\partial \nu}=\frac{\partial v}{\partial \nu}=\frac{\partial w}{\partial \nu}=0},\quad
x\in \partial\Omega, t>0,\\
\disp{u(x,0)=u_0(x)},\tau v(x,0)=\tau v_0(x),w(x,0)=w_0(x),\quad
x\in \Omega,\\
 \end{array}\right.\label{1.1}
\end{equation}
where $\tau>0,\eta\geq0,$
$\Omega\subseteq \mathbb{R}^N, N \geq1$ is the physical domain which we assume to be bounded with
smooth boundary, $\chi$,  $\xi$, $\mu$ and $\eta$ measure the chemotactic sensitivities and haptotactic sensitivities,
 the proliferation rate of the cells and the remodeling rate of the extracellular matrix (ECM), respectively.
%
Here
 the
unknown quantities $u = u(x, t)$, $v = v(x, t)$ and $w = w(x, t)$ denote the density of
cancer cells, the concentration of enzyme and the density of healthy tissue, respectively.

The model \dref{1.1}  accounts for both chemotactic migration
of cancer cells towards a diffusible matrix-degrading enzyme (MDE) secreted
by themselves, and haptotactic migration towards a static tissue, also referred to as
 ECM (Chaplain et al. \cite{Chaplain3,Gerisch},
Liotta and  T. Clair \cite{Liotta22444710}).
This on the one hand opens new fields of applications
to modeling approaches in the style pursued by Keller and Segel (see \cite{Keller79}), but on the other hand it gives
rise to new mathematical challenges due to more involved couplings.



If $\chi = 0$, the PDE system \dref{1.1} becomes the haptotaxis-only system (with remodeling of non-diffusible attractant)
\begin{equation}
 \left\{\begin{array}{ll}
  u_t=\Delta u-\xi\nabla\cdot
  (u\nabla w)+\mu u(1-u-w),\quad
x\in \Omega, t>0,\\
 \disp{v_t=\Delta v +u- v},\quad
x\in \Omega, t>0,\\
\disp{w_t=- vw +\eta w(1-u-w)},\quad
x\in \Omega, t>0.\\
 \end{array}\right.\label{sswxxddff1.1}
\end{equation}
Global existence and asymptotic behavior of solutions to \dref{sswxxddff1.1}  have been investigated in
 \cite{Corriasx319,Corriasx3191,Walker,Lianu1,Tao79477ddffvg,Marciniak,Tao1}
and \cite{Taossdd79477ddffvg} for the case $\eta=0$ and $\eta>0$, respectively.

When $\eta = 0$, the PDE system \dref{1.1} is reduced to  the chemotaxis-haptotaxis system
%
\begin{equation}
 \left\{\begin{array}{ll}
  u_t=\Delta u-\chi\nabla\cdot(u\nabla v)-\xi\nabla\cdot
  (u\nabla w)+\mu u(1-u-w),\quad
x\in \Omega, t>0,\\
 \disp{\tau v_t=\Delta v +u- v},\quad
x\in \Omega, t>0,\\
\disp{w_t=- vw },\quad
x\in \Omega, t>0.\\
 \end{array}\right.\label{sxdcfvvghhh1.1}
\end{equation}
When $\tau= 0$  denotes that  the diffusion rate of the MDE is much greater than that of
cancer cells (see Chaplain and Lolas \cite{Chaplain3444},
Winkler et al. \cite{Bellomo1216,Taox26}).
In \cite{Tao2}, Tao and Wang
 proved that model \dref{sxdcfvvghhh1.1} possesses a unique global bounded classical solution for
any $\mu > 0$ in two space dimensions, and for large $\mu > 0$ in three space dimensions;
Tao and Winkler (\cite{Taox26216}) studied global boundedness for model \dref{sxdcfvvghhh1.1}
with 
the condition $\mu > \frac{(N-2)^{+}}{N}\chi$, furthermore,  they gave the exponential decay of $w$ in the large time limit for the additional explicit
smallness on $w_0$;
While, if $\tau=1$ in \dref{1.1}, Tao and Wang (\cite{Tao3}) proved
that model \dref{1.1} possesses a unique global-in-time classical solution for any $\chi > 0$
in one space dimension, and for small
$\frac{\chi}{\mu}> 0$ in two and three space dimensions;
 Tao (\cite{Tao1}) improved the result of \cite{Tao3} for any $\mu> 0$ in two space dimension;
 Additionally, recent studies have shown that
 the solution behavior can be also impacted by  the nonlinear diffusion (see
  Tao and  Winkler \cite{Tao72}, Wang et al.  \cite{Liughjj791,Wangscd331629,Zhengsssssssddxx}) and the (generalized) logistic damping (see Cao \cite{Cao},
Hillen  et al. \cite{Hillensxdc79}, Zheng et al. \cite{Zhengssddxx,Zhengaassssssssssddxx}).
 Compared with the chemotaxis-only system, chemotaxis-only system and the chemotaxis-haptotaxis system,
  the coupled chemotaxis-haptotaxis system   with remodeling of non-diffusible attractant ($\eta>0$ in  \dref{1.1}) is much less understood (Chaplain and Lolas \cite{Chaplain3444},
 Pang and  Wang \cite{PangPang1},
Tao and  Winkler \cite{Tao79477}). The main technical difficulty in their proof stems from the effects of the strong coupling in  \dref{1.1}  on the spatial regularity of $u,$ $v$ and $w$
when $\eta>0.$ When $\eta=0,$ one can 
build
a one-sided pointwise estimate  which connects $\Delta w$  to $v$ (see Lemma 2.2 of \cite{Cao} or (3.10) of \cite{Wangscd331629}). Relying on such a
pointwise estimate, we can derive two useful energy-type inequalities that bypass $\int_{\Omega} u^{p-1}\nabla\cdot( u \nabla w)$ (see Lemma 3.2 of \cite{Zhengsssssssddxx}). Using such information along with coupled estimate techniques and the boundedness of the $\|\nabla v(\cdot, t)\|_{L^2(\Omega)}$, we establish
estimates on
$\int_{\Omega} u^p+|\nabla v|^{2q}$
 for any $p$ and $q > 1$ (see Lemmata  3.3 and 3.4 of \cite{Zhengsssssssddxx}), which results in the boundedness of $u$ in $L^\infty(\Omega)$ by using the standard regularity theory of parabolic
equation and performing the Moser iteration procedure (see Lemma  3.5 of \cite{Zhengsssssssddxx}).
%
%
%
However, for the model  \dref{1.1} with $\eta>0$, one needs to estimate the chemotaxis-related integral term
$\int_{\Omega}a^p|\nabla v|^2dx$ (see (3.28) in \cite{Tao79477}) or $\int_{\Omega}e^{-(p+1)(t-s)}a^p|\nabla v|^2dxds$
(see (3.8) of  \cite{PangPang1}) with $a:= ue^{-\xi w}$, which proves to be much more technically demanding. 
In \cite{PangPang1}, assuming that  $\mu>\xi\eta\max\{\|u_0\|_{L^\infty(\Omega)},1\}+\mu^*(\chi^2,\xi)$ (the hypothesis can not be dropped (see the proof of Lemma 3.2 of \cite{PangPang1})), Pang and Wang showed that the problem \dref{1.1}
admits a unique global solution $(u,v,w)\in (C^{2,1}
(\bar{\Omega}\times(0,\infty)))^3$. Moreover, $u$ is bounded in $\Omega\times(0,\infty)$. However, to
the best of our knowledge, it is still an open problem to determine whether or not
in the case $N=2$ some unbounded solutions may exist in \dref{1.1} with {\bf small} $\mu>0$.
%
Indeed, as pointed by \cite{Bellomo1216}  (see also \cite{Taox3201}), the hypothesis on $\mu>0$ may yield the  classical global solution.
 So, it is natural to ask whether the solution is globally existence 
when $\mu> 0$. In this paper, we give a positive answer to this question. 
%


Motivated by the aforementioned papers, the purpose of this work is to establish global solvability of \dref{1.1}.
%
%
%
%
%
Our main
result in this respect reads as follows.
\begin{theorem}\label{theorem3}
Let $\tau>0,\chi >0, \xi >0$ and $\eta > 0.$
Assume
that $\Omega\subseteq \mathbb{R}^2$ is a bounded    domain with smooth boundary and
the initial data $(u_0, v_0,w_0)$ is supposed to satisfy the following
conditions
\begin{equation}\label{x1.731426677gg}
\left\{
\begin{array}{ll}
\displaystyle{u_0\in C^{2+\vartheta}(\bar{\Omega})~~\mbox{with}~~u_0\geq0~~\mbox{in}~~\Omega~~\mbox{and}~~\frac{\partial u_0}{\partial\nu}=0~~\mbox{on}~~\partial\Omega},\\
\displaystyle{v_0\in C^{2+\vartheta}(\bar{\Omega})~~\mbox{with}~~v_0\geq0~~\mbox{in}~~\Omega~~\mbox{and}~~\frac{\partial v_0}{\partial\nu}=0~~\mbox{on}~~\partial\Omega},\\
\displaystyle{w_0\in C^{2+\vartheta}(\bar{\Omega})~~\mbox{with}~~w_0\geq0~~\mbox{in}~~\bar{\Omega}~~\mbox{and}~~\frac{\partial w_0}{\partial\nu}=0~~\mbox{on}~~\partial\Omega} ~~~~~~~~~~~~~~~~~~~~~~~~~~~~~~~~~~~~~~~~~~~~~~~~~~~~~~~~~~~~~~~~~~~\\
\end{array}
\right.
\end{equation}
with some $\vartheta\in(0,1).$
If $\mu>0$,
%
%
%
%
%
%
then there exists a triple $(u,v,w)\in (C^{2,1}
(\bar{\Omega}\times(0,\infty)))^3$ which solves \dref{1.1} in the classical sense. Moreover, $u$ and $v$  are bounded in $\Omega\times(0,\infty)$.
\end{theorem}
\begin{remark}
%
%


 (i) If $w\equiv0$, (the PDE system \dref{1.1} is reduced to the chemotaxis-only system), it
is not difficult to obtain that the solutions under the conditions of Theorem \ref{theorem3} are
uniformly bounded when $N=2$, which  coincides with the results of   Osaki et al. (\cite{Osakix391}).

 (ii) From Theorem \ref{theorem3}, we derive that  solutions of model \dref{1.1}
are global and bounded for any $\eta=0,\mu>0$ and $N\leq2$, which coincides with the result of  Tao (\cite{Taox3201}).

%

%
\end{remark}
Without loss of generality, we may assume $\tau=1$ in \dref{1.1}, since, for $\tau>0$ can be proved very similarly.

The plan of this paper is as follows. In Section 2, we give some basic results and some preliminary
lemmata  as a preparation for the arguments in the later sections. In Section 3, firstly, by using the technical lemma (Lemma \ref{lemma45630223116}) and employing the variation-of-constants formula, we may establish the boundedness of
 $\int_{\Omega}{a^{q_0}}(q_0 > 1),$ where $a=ue^{-\xi w}.$
 In addition, we shall involve the variation-of-constants formula and $L^p$-estimate techniques
to gain
%
%
%
the boundedness of $\int_{\Omega}a^{p}(p>1)$.
Finally, using  the Alikakos--Moser iteration, we finally established the $L^\infty(\Omega)$ bound of $a$ (see the proof of  Theorem \ref{theorem3}).

\section{Preliminaries}
Before formulating our main results, we first recall some preliminary
lemmas used
throughout this paper.
%
To begin with, let us collect some basic solution properties which essentially have already been used
in \cite{Horstmann791} (see also Winkler \cite{Winkler792}, Zhang and Li \cite{Zhangddff4556}).
\begin{lemma}(\cite{Horstmann791})\label{lemmaggbb41ffgg}
For $p\in(1,\infty)$, let $A := A_p$ denote the sectorial operator defined by
\begin{equation}\label{hnjmkfgbhnn6291}
A_pu :=-\Delta u~~\mbox{for all}~~u\in D(A_p) :=\{\varphi\in W^{2,p}(\Omega)|\frac{\partial\varphi}{\partial \nu}|_{\partial\Omega}=0\}.
\end{equation}
The operator $A + 1$ possesses fractional powers $(A + 1)^{\alpha}(\alpha\geq0)$, the domains of
which have the embedding properties
\begin{equation}\label{hnjmkfgbhnn6291}
D((A+1)^\alpha)\hookrightarrow W^{1,p}(\Omega)~~\mbox{if}~~\alpha>\frac{1}{2}.
\end{equation}

If $m\in \{0, 1\}$, $p\in [1,\infty]$ and $q \in (1,\infty)$ with $m-\frac{N}{p} < 2\alpha-\frac{N}{q} $, then we have
\begin{equation}\label{hnssedrrffjmkfgbhnn6291}
\|u\|_{W^{m,p}(\Omega)}\leq C\|(A+1)^\alpha u\|_{L^{q}(\Omega)}~~~\mbox{for all}~~u\in D((A+1)^\alpha),
\end{equation}
 where $C$ is a positive constant.
The fact that the spectrum of $A$ is a $p$-independent countable set of positive real numbers
$0 = \mu_0 < \mu_1 <\mu_2 <\cdots $
entails the following consequences:
For all $1\leq p < q < \infty$ and $u\in L^p(\Omega)$ the
general $L^p$-$L^q$ estimate
\begin{equation}\label{gbhnhnjmkfgbhnn6291}
\|(A+1)^\alpha e^{-tA}u\|_{L^q(\Omega)}\leq ct^{-\alpha-\frac{N}{2}(\frac{1}{p}-\frac{1}{q})}e^{(1-\mu)t}\|u\|_{L^p(\Omega)}
\end{equation}
for any $t > 0$ and
$\alpha\geq0$ with some 	$\mu > 0.$
\end{lemma}

In deriving some preliminary estimates for $v$, we shall make use of  following the property referred to as
a variation of Maximal Sobolev Regularity (see e.g. Theorem 3.1 of \cite{Hieber} or \cite{Cao}).
\begin{lemma}\label{lemma45xy1222232}
Suppose  $\gamma\in (1,+\infty)$, $g\in L^\gamma((0, T); L^\gamma(
\Omega))$. 
Let $v$ be a solution of the following initial boundary value
 \begin{equation}
 \left\{\begin{array}{ll}
v_t -\Delta v+v=g,~~~(x, t)\in
 \Omega\times(0, T ),\\
\disp\frac{\partial v}{\partial \nu}=0,~~~(x, t)\in
 \partial\Omega\times(0, T ),\\
v(x,0)=v_0(x),~~~(x, t)\in
 \Omega.\\
 \end{array}\right.\label{1.3xcx29}
\end{equation}
Then there exists a positive constant $C_\gamma$ such that if $s_0\in[0,T)$, $v(\cdot,s_0)\in W^{2,\gamma}(\Omega)(\gamma>N)$ with $\disp\frac{\partial v(\cdot,s_0)}{\partial \nu}=0,$ then,
\begin{equation}
\begin{array}{rl}
&\disp{\int_{s_0}^Te^{\gamma s}\| v(\cdot,t)\|^{\gamma}_{W^{2,\gamma}(\Omega)}ds\leq C_\gamma\left(\int_{s_0}^Te^{\gamma s}
\|g(\cdot,s)\|^{\gamma}_{L^{\gamma}(\Omega)}ds+e^{\gamma s_0}(\|v_0(\cdot,s_0)\|^{\gamma}_{W^{2,\gamma}(\Omega)})\right).}\\
\end{array}
\label{cz2.5bbv114}
\end{equation}
\end{lemma}
\begin{proof}
Letting $c(x, s) = e^{s}v(x, s)$. Then we derive that $c$ satisfies
\begin{equation}
 \left\{\begin{array}{ll}
c_s(x,s) -\Delta c(x,s)=f(x,s),~~~(x, s)\in
 \Omega\times(0, T ),\\
\disp\frac{\partial c}{\partial \nu}=0,~~~(x, s)\in
 \partial\Omega\times(0, T ),\\
c(x,0)=v_0(x),~~~(x, s)\in
 \Omega,\\
 \end{array}\right.\label{1.3ssderrttgxcx29}
\end{equation}
where $f(x,s)=e^sg(x,s).$
Applying the Maximal Sobolev Regularity (see e.g. Theorem 3.1 of \cite{Hieber}) to $c$, we derive that
\begin{equation}
\begin{array}{rl}
&\disp{\int_{0}^T\|\Delta c(\cdot,s)\|^{\gamma}_{L^{\gamma}(\Omega)}ds+\int_{0}^T\| c(\cdot,s)\|^{\gamma}_{L^{\gamma}(\Omega)}ds+\int_{0}^T\| c_s(\cdot,s)\|^{\gamma}_{L^{\gamma}(\Omega)}ds}\\
\leq &\disp{C_{1,\gamma}\left(\int_{0}^T
\|f(\cdot,s)\|^{\gamma}_{L^{\gamma}(\Omega)}ds+(\|c_0\|^{\gamma}_{L^{\gamma}(\Omega)}+\|\Delta c_0\|^{\gamma}_{L^{\gamma}(\Omega)})\right).}\\
\end{array}
\label{cz2.5bbvssdefff114}
\end{equation}
Substituting $v$ into the above inequality and changing the variables imply
\begin{equation}
\begin{array}{rl}
&\disp{\int_{0}^Te^{\gamma s}(\| v(\cdot,t)\|^{\gamma}_{L^{\gamma}(\Omega)}+\|\Delta v(\cdot,t)\|^{\gamma}_{L^{\gamma}(\Omega)})ds}\\
\leq &\disp{C_{1,\gamma}\left(\int_{0}^T
e^{\gamma s}
\|g(\cdot,s)\|^{\gamma}_{L^{\gamma}(\Omega)}ds+(\|c_0\|^{\gamma}_{L^{\gamma}(\Omega)}+\|\Delta c_0\|^{\gamma}_{L^{\gamma}(\Omega)})\right).}\\
\end{array}
\label{cz2.5bbvssdefssderffff114}
\end{equation}
On the other hand, by the elliptic $L^p$-estimate,
\begin{equation}
\|v\|_{W^{2,\gamma}(\Omega)}\leq C_{2,\gamma}(\|\Delta v\|_{L^\gamma(\Omega)}+
\|\Delta v\|_{L^\gamma(\Omega)})~~~\mbox{for any}~~ v\in W^{2,\gamma}(\Omega)~~~\mbox{with}~~ \frac{\partial v}{\partial\nu}=0.
\label{cz2.5bbvssdefssdsderrrerffff114}
\end{equation}
Consequently, combining   \dref{cz2.5bbvssdefssderffff114}
with \dref{cz2.5bbvssdefssdsderrrerffff114},  for any $s_0 > 0,$ replacing $v(t)$ by $v(t + s_0)$, we derive  \dref{cz2.5bbv114}.
\end{proof}

The Young inequality (\cite{Ladyzenskaja710}): Let $1 < p, q < +\infty$,
$\frac{1}{p}+
\frac{1}{q}
= 1$. Then  for any positive constants $a$ and $b$, we have
$$ab\leq
\varepsilon a^p
+
\frac{1}{q}(\varepsilon p)^{-\frac{q}{p}}b^q.$$

The following lemma deals with local-in-time existence and uniqueness of a classical solution for the
problem \dref{1.1} (see \cite{PangPang1}).
\begin{lemma}\label{lemma70} (\cite{PangPang1})
Assume that the nonnegative functions $u_0,v_0,$ and $w_0$ satisfies \dref{x1.731426677gg}
for some $\vartheta\in(0,1).$
%
%
 Then there exists a maximal existence time $T_{max}\in(0,\infty]$ and a triple of  nonnegative functions 
 $$
\left\{\begin{array}{rl}
&\disp{a\in C^0(\bar{\Omega}\times[0,T_{max}))\cap C^{2,1}(\bar{\Omega}\times(0,T_{max})),}\\
&\disp{v\in C^0(\bar{\Omega}\times[0,T_{max}))\cap C^{2,1}(\bar{\Omega}\times(0,T_{max})),}\\
&\disp{w\in  C^{2,1}(\bar{\Omega}\times[0,T_{max})),}\\
\end{array}\right.
$$
 which solves \dref{1.1}  classically and satisfies
 \begin{equation}0\leq w\leq \rho:=\max\{1,\|w_0\|_{L^\infty(\Omega)}\}
  ~~\mbox{in}~~ \Omega\times(0,T_{max}).
  \label{1.1sddd6ssdddd3072x}
\end{equation}
Moreover, if  $T_{max}<+\infty$, then
\begin{equation}
\|a(\cdot, t)\|_{L^\infty(\Omega)}+\|\nabla w(\cdot,t)\|_{L^{5}(\Omega)}\rightarrow\infty~~ \mbox{as}~~ t\nearrow T_{max}.
\label{1.163072x}
\end{equation}
\end{lemma}

Firstly, by Lemma \ref{lemma70}, we can  pick  $s_0\in(0,T_{max})$, $s_0\leq1$ and
$\beta>0$ such that
\begin{equation}\label{eqx45xx12112}
\|u(\tau)\|_{L^\infty(\Omega)}\leq \beta,~~~\|v(\tau)\|_{W^{1,\infty}(\Omega)}\leq \beta~~\mbox{and}~~\|w(\tau)\|_{W^{2,\infty}(\Omega)}\leq \beta~~\mbox{for all}~~\tau\in[0,s_0].
\end{equation}
In some parts of our subsequent analysis,  we introduce the variable transformation (see Tao et al. \cite{Tao2,Tao72,Tao79477}, Pang and Wang \cite{PangPang1})
\begin{equation}\label{eqx45xx1ssdee2112}
a=ue^{-\xi w},
\end{equation}
upon which \dref{1.1} takes the form
\begin{equation}
 \left\{\begin{array}{ll}
  a_t=e^{-\xi w}\nabla\cdot(e^{\xi w}\nabla a)-\chi e^{-\xi w}\nabla\cdot(e^{\xi w}a\nabla v)+\xi avw+ a(\mu-\xi\eta w)(1-e^{\xi w}a-w),~
x\in \Omega, t>0,\\
 \disp{v_t=\Delta v +ae^{\xi w}- v},\quad
x\in \Omega, t>0,\\
\disp{w_t=- vw+\eta w(1-ae^{\xi w}-w) },\quad
x\in \Omega, t>0,\\
 \disp{\frac{\partial a}{\partial \nu}=\frac{\partial v}{\partial \nu}=\frac{\partial w}{\partial \nu}=0},\quad
x\in \partial\Omega, t>0,\\
\disp{a(x,0):=a_0(x)=u_0(x)e^{-\xi w_0(x)}},v(x,0)=v_0(x),w(x,0)=w_0(x),\quad
x\in \Omega.\\
 \end{array}\right.\label{ghyyyuuu1.1}
\end{equation}

\section{Proof of the main result}

In this section, we are going to establish an iteration step to develop the main ingredient of our result.
Firstly, based on the ideas of Lemma 3.1  in \cite{PangPang1} (see also Lemma 2.1 of \cite{Winkler37103}), we
can derive the following  properties of solutions of  \dref{1.1}.
%
%
%

\begin{lemma}\label{wsdelemma45}
Under the assumptions in theorem \ref{theorem3}, we derive that
there exists a positive constant 
$C$
such that the solution of \dref{1.1} satisfies
%
%
\begin{equation}
\int_{\Omega}{u(x,t)}+\int_{\Omega} {v^2}(x,t)+\int_{\Omega}|\nabla {v}(x,t)|^2 \leq C~~\mbox{for all}~~ t\in(0, T_{max}).
\label{cz2.5ghju48cfg924ghyuji}
\end{equation}
Moreover,
for each $T\in(0, T_{max})$, one can find a constant $C > 0$ independent of $\varepsilon$ such that
\begin{equation}
\begin{array}{rl}
&\disp{\int_{0}^T\int_{\Omega}[|\nabla {v}|^2+u^2+ |\Delta {v}|^2]\leq C.}\\
\end{array}
\label{bnmbncz2.5ghhjuyuivvbssdddeennihjj}
\end{equation}
\end{lemma}

\begin{lemma}\label{lemma45630223116}
Let \begin{equation}
{A}_1=\frac{1}{\delta+1}(\frac{\delta+1}{\delta})^{-\delta}
[\frac{\delta(\delta-1)}{2}\chi^2]^{\delta+1}C_7C_{\delta+1}e^{\xi(\delta-1)}
\label{zjscz2.5297x9630222211444125}
\end{equation}
and $H(y)=y+
{A}_1y^{- \delta }$ for $y>0.$
For any fixed $\delta\geq1,C_7,\chi,C_{\delta+1}>0,$
then $$\min_{y>0}H(y)=\frac{\delta(\delta-1)\chi^2}{2}(C_7C_{\delta+1})^{\frac{1}{\delta+1}}.$$
\end{lemma}
\begin{proof}
It is easy to verify that $$H'(y)=1- A_1\delta y^{-\delta-1}.$$
Let $H'(y)=0$, we have
$$y=\left(A_1\delta\right)^{\frac{1}{\delta+1}}.$$
On the other hand, by $\lim_{y\rightarrow0^+}H(y)=+\infty$ and $\lim_{y\rightarrow+\infty}H(y)=+\infty$, we have
$$\begin{array}{rl}
\min_{y>0}H(y)=H[\left(A_1\delta\right)^{\frac{1}{\delta+1}}]=&\disp{\frac{\delta(\delta-1)\chi^2}{2}(C_7C_{\delta+1})^{\frac{1}{\delta+1}},}\\
\end{array}
$$
whereby the proof is completed.
\end{proof}

\begin{lemma}\label{lemma45630223}
Let $\mu,\chi,\eta$ and $\xi$ be the positive constants. Assuming that
 $(a,v,w)$ is  a solution to \dref{ghyyyuuu1.1} on $(0,T_{max})$.
Then 
 for all $p>1$,
there exists a positive constant $C:=C(p,|\Omega|,\mu,\chi,\xi,\eta,\beta)$ such that 
\begin{equation}
\int_{\Omega}a^p(x,t)dx\leq C ~~~\mbox{for all}~~ t\in(0,T_{max}).
\label{zjscz2.5297x96302222114}
\end{equation}
\end{lemma}
\begin{proof}
Firstly, assuming that  $p\leq2$.
A straightforward differentiation, using \dref{ghyyyuuu1.1} and two integrations by parts yields
\begin{equation}
\begin{array}{rl}
&\disp\frac{d}{dt}\disp\int_{\Omega}e^{\xi w}a^p+(p+1)\int_{\Omega}e^{\xi w}a^p\\
=&\disp{\xi\int_{\Omega}e^{\xi w}a^p\cdot\{-vw+\eta w(1-ae^{\xi w}-w)\}}\\
&+\disp{p\int_\Omega e^{\xi w}a^{p-1}\cdot\{e^{-\xi w}\nabla\cdot(e^{\xi w}\nabla a)-\chi e^{-\xi w}\nabla\cdot(e^{\xi w}a\nabla v)\}}\\
&+\disp{a\xi vw +a(\mu-\xi\eta w)(1-ae^{\xi w}-w)\}+(p+1)\int_{\Omega}e^{\xi w}a^p}\\
=&\disp{-p(p-1)\int_{\Omega}
e^{\xi w}a^{p-2}|\nabla a|^2+p(p-1)\chi\int_{\Omega}
e^{\xi w}a^{p-1}\nabla a\cdot \nabla v}\\
&+\disp{(p-1)\xi\int_\Omega e^{\xi w}a^{p}vw+\int_\Omega e^{\xi w}a^{p}\{(p+1)+(p-1)\xi\eta w(w-1)+p\mu (1-w)\}}\\
&+\disp{\int_\Omega e^{2\xi w}a^{p+1}[(p-1)\xi\eta w-p\mu]}\\
:=&\disp{J_1+J_2+J_3+J_4+J_5~~\mbox{for all}~~ t\in(0,T_{max}).}\\
\end{array}
\label{cz2.511411ssedd4}
\end{equation}
Now, in light of \dref{1.1sddd6ssdddd3072x} and the Young inequality, we derive that
\begin{equation}
\begin{array}{rl}
\disp J_3\leq&\disp{\varepsilon_1\int_\Omega e^{2\xi w}a^{p+1}+\frac{1}{p+1}(\varepsilon_1\times\frac{p+1}{p})^{-p}
[(p-1)\xi]^{p+1}\int_\Omega e^{\xi w(1-p)}v^{p+1}}\\
\leq&\disp{\varepsilon_1\int_\Omega e^{2\xi w}a^{p+1}+\frac{1}{p+1}(\varepsilon_1\times\frac{p+1}{p})^{-p}
[(p-1)\xi]^{p+1}\int_\Omega v^{p+1}~~\mbox{for all}~~ t\in(0,T_{max}),}\\
\end{array}
\label{cz2.511411ssedssderrd4}
\end{equation}
\begin{equation}
\begin{array}{rl}
\disp J_4\leq&\disp{[(p+1)+(p-1)\xi\eta\rho^2+p\mu]\int_\Omega e^{\xi w}a^{p}}\\
\leq&\disp{(p+1)[1+\xi\eta\rho^2+\mu]\int_\Omega e^{\xi w}a^{p}}\\
\leq&\disp{\varepsilon_2\int_\Omega e^{2\xi w}a^{p+1}+\frac{1}{p+1}(\varepsilon_2\times\frac{p+1}{p})^{-p}
(p+1)^{p+1}[1+\xi\eta\rho^2+\mu]^{p+1}|\Omega|~~\mbox{for all}~~ t\in(0,T_{max}),}\\
\end{array}
\label{cz2.511411ssedssderrd4}
\end{equation}
\begin{equation}
\begin{array}{rl}
\disp J_5\leq&\disp{\int_\Omega e^{2\xi w}a^{p+1}[(p-1)\xi\eta \rho-p\mu]~~\mbox{for all}~~ t\in(0,T_{max})}\\
\end{array}
\label{cz2.511411sdfffssedd4}
\end{equation}
and
\begin{equation}
\begin{array}{rl}
\disp J_2\leq&\disp{\frac{p(p-1)}{2}\int_{\Omega}
e^{\xi w}a^{p-2}|\nabla a|^2+\frac{p(p-1)}{2}\chi^2\int_{\Omega}
e^{\xi w}a^{p}|\nabla v|^2}\\
\leq&\disp{\frac{p(p-1)}{2}\int_{\Omega}
e^{\xi w}a^{p-2}|\nabla a|^2+\lambda_0\int_\Omega e^{2\xi w}a^{p+1}}\\
&\disp{+\frac{1}{p+1}(\lambda_0\times\frac{p+1}{p})^{-p}
[\frac{p(p-1)}{2}\chi^2]^{p+1}\int_\Omega e^{(1-p)\xi w}|\nabla v|^{2(p+1)} }\\
\leq&\disp{\frac{p(p-1)}{2}\int_{\Omega}
e^{\xi w}a^{p-2}|\nabla a|^2+\lambda_0\int_\Omega e^{2\xi w}a^{p+1}}\\
&\disp{+\frac{1}{p+1}(\lambda_0\times\frac{p+1}{p})^{-p}
[\frac{p(p-1)}{2}\chi^2]^{p+1}\int_\Omega |\nabla v|^{2(p+1)} ~~\mbox{for all}~~ t\in(0,T_{max})}\\
\end{array}
\label{cz2.511411ssddfffffedssderrd4}
\end{equation}
and any small positive constants $\varepsilon_1,\varepsilon_2$ and $\lambda_0.$

Inserting \dref{cz2.511411ssedssderrd4}--\dref{cz2.511411ssddfffffedssderrd4} into \dref{cz2.511411ssedd4}, we derive that
\begin{equation}
\begin{array}{rl}
\disp &\disp\frac{d}{dt}\disp\int_{\Omega}e^{\xi w}a^p+(p+1)\int_{\Omega}e^{\xi w}a^p+\int_\Omega e^{2\xi w}a^{p+1}[p\mu-\varepsilon_1-\varepsilon_2-\lambda_0-(p-1)\xi\eta \rho]\\
\leq&\disp{\frac{1}{p+1}(\lambda_0\times\frac{p+1}{p})^{-p}
[\frac{p(p-1)}{2}\chi^2]^{p+1}\int_\Omega |\nabla v|^{2(p+1)}}\\
&\disp{+C_1(\varepsilon_1,\varepsilon_2)~~\mbox{for all}~~ t\in(0,T_{max}),}\\
\end{array}
\label{cz2.511411ssssdffedd4}
\end{equation}
where
\begin{equation}
\begin{array}{rl}
C_1(\varepsilon_1,\varepsilon_2):=&\disp{\frac{1}{p+1}(\varepsilon_2\times\frac{p+1}{p})^{-p}(p+1)^{p+1}
[1+\xi\eta\rho^2+\mu]^{p+1}|\Omega|}\\
&\disp{+\frac{1}{p+1}(\varepsilon_1\times\frac{p+1}{p})^{-p}
[(p-1)\xi]^{p+1}\int_\Omega v^{p+1}.}\\
\end{array}
\label{cz2.511411ssdeerrssssdffedd4}
\end{equation}
Next,
from Lemma \ref{wsdelemma45}, $N=2$ and the Gagliardo--Nirenberg inequality, it follows that
\begin{equation}
\|v(\cdot,t)\|_{L^p(\Omega)}\leq C_2~~\mbox{for all}~~ p\geq1~~\mbox{and}~~t\in(0,T_{max}).
\label{cz2.511411shhyuuuusdeerrssssdffedd4}
\end{equation}
This along with \dref{cz2.511411ssdeerrssssdffedd4} entails
\begin{equation}
\begin{array}{rl}
C_1(\varepsilon_1,\varepsilon_2)\leq&C_3(\varepsilon_1,\varepsilon_2)\\
:=&\disp{\frac{1}{p+1}(\varepsilon_2\times\frac{p+1}{p})^{-p}
(p+1)^{p+1}
[1+\xi\eta\rho^2+\mu]^{p+1}|\Omega|}\\
&\disp{+C_2\frac{1}{p+1}(\varepsilon_1\times\frac{p+1}{p})^{-p}
[(p-1)\xi]^{p+1}.}\\
\end{array}
\label{cz2.511411ssdeessddrrssssdffedd4}
\end{equation}
From this and \dref{cz2.511411ssdeerrssssdffedd4} we also obtain
\begin{equation}
\begin{array}{rl}
\disp &\disp\frac{d}{dt}\disp\int_{\Omega}e^{\xi w}a^p+(p+1)\int_{\Omega}e^{\xi w}a^p+\int_\Omega e^{2\xi w}a^{p+1}[p\mu-\varepsilon_1-\varepsilon_2-\lambda_0-(p-1)\xi\eta \rho]\\
\leq&\disp{\frac{1}{p+1}(\lambda_0\times\frac{p+1}{p})^{-p}
[\frac{p(p-1)}{2}\chi^2]^{p+1}\int_\Omega |\nabla v|^{2(p+1)}+C_3(\varepsilon_1,\varepsilon_2)~~\mbox{for all}~~ t\in(0,T_{max}).}\\
\end{array}
\label{cz2.511411sssdffsssdffedd4}
\end{equation}
Then for any $t\in (s_0,T_{max})$, by means of the variation-of constants representation for the above inequality, we can estimate
%
\begin{equation}
\begin{array}{rl}
\disp &\disp\disp\int_{\Omega}e^{\xi w}a^p(\cdot,t)+[p\mu-\varepsilon_1-\varepsilon_2-\lambda_0-(p-1)\xi\eta \rho]\int_{s_0}^t\int_\Omega e^{-(p-1)(t-s)}e^{2\xi w}a^{p+1}\\
\leq&\disp{\int_{\Omega}u^p(s_0,t)+\frac{1}{p+1}(\lambda_0\times\frac{p+1}{p})^{-p}
[\frac{p(p-1)}{2}\chi^2]^{p+1}\int_{s_0}^t\int_\Omega e^{-(p-1)(t-s)}|\nabla v|^{2(p+1)}}\\
&\disp{+C_3(\varepsilon_1,\varepsilon_2)~~\mbox{for all}~~ t\in(0,T_{max}).}\\
\end{array}
\label{cz2.511411sssdffsddffrtssdffedd4}
\end{equation}
%
%
%
Next, according to the Gagliardo--Nirenberg inequality, \dref{cz2.511411shhyuuuusdeerrssssdffedd4} and  Lemma \ref{wsdelemma45}, we can choose $C_4$ and $C_5$ such that
\begin{equation}
\begin{array}{rl}
\disp \|\nabla v(\cdot,s)\|^{2(p+1)}_{L^{2(p+1)}(\Omega)}\leq&\disp{ C_4\| v(\cdot,s)\|^{p+1}_{W^{2,p+1}(\Omega)} \|\nabla v(\cdot,s)\|^{p+1}_{L^{2}(\Omega)}}\\
\leq&\disp{C_5\| v(\cdot,s)\|^{p+1}_{W^{2,p+1}(\Omega)}  ~~\mbox{for all}~~ t\in(0,T_{max}).}\\
\end{array}
\label{cz2.511411ssddssddfffffffedssderrd4}
\end{equation}
Therefore, due to $p\leq2,$ with the help of \dref{cz2.511411ssddssddfffffffedssderrd4}, applying \dref{cz2.5bbv114} of Lemma \ref{lemma45xy1222232} with $\gamma=p+1$, we obtain
\begin{equation}
\begin{array}{rl}
\disp &\disp\disp\frac{1}{p+1}(\lambda_0\times\frac{p+1}{p})^{-p}
[\frac{p(p-1)}{2}\chi^2]^{p+1}\int_{s_0}^t\int_\Omega e^{-(p-1)(t-s)}|\nabla v|^{2(p+1)}\\
\leq&\disp{\frac{1}{p+1}(\lambda_0\times\frac{p+1}{p})^{-p}
[\frac{p(p-1)}{2}\chi^2]^{p+1}C_5\int_{s_0}^t e^{-(p-1)(t-s)}\| v(\cdot,s)\|^{p+1}_{W^{2,p+1}(\Omega)}ds}\\
\leq&\disp{\frac{1}{p+1}(\lambda_0\times\frac{p+1}{p})^{-p}
[\frac{p(p-1)}{2}\chi^2]^{p+1}C_5C_{p+1}\int_{s_0}^t \int_\Omega e^{-(p-1)(t-s)} u^{{{p}+1}}(x,s)dxds+
C_6 }\\
\leq&\disp{\frac{1}{p+1}(\lambda_0\times\frac{p+1}{p})^{-p}
[\frac{p(p-1)}{2}\chi^2]^{p+1}C_5C_{p+1}e^{\xi (p-1)}\int_{s_0}^t \int_\Omega e^{-(p-1)(t-s)} e^{2\xi w}a^{p+1}(x,s)dxds+
C_6 }\\
\leq&\disp{\frac{1}{p+1}(\lambda_0\times\frac{p+1}{p})^{-p}
[\frac{p(p-1)}{2}\chi^2]^{p+1}C_7C_{p+1}\int_{s_0}^t \int_\Omega e^{-(p-1)(t-s)} e^{2\xi w}a^{p+1}(x,s)dxds+
C_6 }\\
\end{array}
\label{cz2.511411ssderrrsssdffsddffrtssdffedd4}
\end{equation}
for all $t\in(s_0,T_{max})$,
where
\begin{equation}
C_6:=\frac{1}{p+1}(\lambda_0\times\frac{p+1}{p})^{-p}
[\frac{p(p-1)}{2}\chi^2]^{p+1}C_5C_{p+1}e^{\gamma s_0}\|v_0(\cdot,s_0)\|^{\gamma}_{W^{2,\gamma}(\Omega)}~~~\mbox{and}~~C_7:=C_5e^{\xi}.
\label{cz2.511411ssseddsderrrsssdffsddffrtssdffedd4}
\end{equation}
Substituting \dref{cz2.511411ssderrrsssdffsddffrtssdffedd4} into \dref{cz2.511411sssdffsddffrtssdffedd4}, we derive
\begin{equation}
\begin{array}{rl}
\disp &\disp\disp\int_{\Omega}e^{\xi w}a^p(\cdot,t)+[p\mu-\varepsilon_1-\varepsilon_2-\lambda_0-(p-1)\xi\eta \rho]\int_{s_0}^t\int_\Omega e^{-(p-1)(t-s)}e^{2\xi w}a^{p+1}\\
\leq&\disp{\frac{1}{p+1}(\lambda_0\times\frac{p+1}{p})^{-p}
[\frac{p(p-1)}{2}\chi^2]^{p+1}C_7C_{p+1}\int_{s_0}^t \int_\Omega e^{-(p-1)(t-s)} e^{2\xi w}a^{p+1}(x,s)dxds+C_8(\varepsilon_1,\varepsilon_2)}\\
\end{array}
\label{cz2.511411sssdffsddfffdffrtssdffedd4}
\end{equation}
for all $t\in(0,T_{max}),$
where
$$C_8(\varepsilon_1,\varepsilon_2):=C_3(\varepsilon_1,\varepsilon_2)+C_6.
$$
Choosing $\lambda_0=\left(A_1p\right)^{\frac{1}{p+1}}$ in \dref{cz2.511411sssdffsddfffdffrtssdffedd4} and using Lemma \ref{lemma45630223116}, we derive
\begin{equation}
\begin{array}{rl}
\disp &\disp\disp\int_{\Omega}e^{\xi w}a^p(\cdot,t)\\
&\disp+[p\mu-\varepsilon_1-\varepsilon_2-\frac{p(p-1)\chi^2}{2}(C_7C_{p+1})^{\frac{1}{p+1}}-(p-1)\xi\eta \rho]\int_{s_0}^t\int_\Omega e^{-(p-1)(t-s)}e^{2\xi w}a^{p+1}\\
\leq&\disp{C_8(\varepsilon_1,\varepsilon_2).}\\
\end{array}
\label{cz2.511411sssdffssddffsddfffdffrtssdffedd4}
\end{equation}
Now, for any positive constants  $\mu,\chi,\xi$ and $\eta$, we may pick  $p_0>1$ which is close to $1$ such that
\begin{equation}
p_0\mu-\frac{p_0(p_0-1)\chi^2}{2}(C_7C_{p_0+1})^{\frac{1}{p_0+1}}-(p_0-1)\xi\eta \rho>0,
\label{cz2ssderrr.511411sssdffssddffsddfffdffrtssdffedd4}
\end{equation}
thus, we can choose  $\varepsilon_1$ and $\varepsilon_2$  appropriately  small such that
 \begin{equation}0<\varepsilon_1+\varepsilon_2<p_0\mu-\frac{p_0(p_0-1)\chi^2}{2}(C_7C_{p_0+1})^{\frac{1}{p_0+1}}-(p_0-1)\xi\eta \rho.
\label{cz2.5kk1214114114ssdddrrttrrgssdeersdddtttgkkll}
\end{equation}
Collecting \dref{cz2.511411sssdffssddffsddfffdffrtssdffedd4} and  \dref{cz2.5kk1214114114ssdddrrttrrgssdeersdddtttgkkll}, we derive that for some $p_0>1$, there exists a positive constant $C_9$
such that
\begin{equation}
\begin{array}{rl}
&\disp{\int_{\Omega}u^{{p_0}}(x,t) dx\leq C_9~~\mbox{for all}~~t\in (s_0, T_{max}).}\\
\end{array}
\label{cz2.5kk1214114114rrggkklljjuu}
\end{equation}
Next,
we fix $q <\frac{2{p_0}}{(2-{p_0})^+}$
and choose some
 $\alpha> \frac{1}{2}$ such that
\begin{equation}
q <\frac{1}{\frac{1}{p_0}-\frac{1}{2}+\frac{2}{2}(\alpha-\frac{1}{2})}\leq\frac{2{p_0}}{(2-{p_0})^+}.
\label{fghgbhnjcz2.5ghju48cfg924ghyuji}
\end{equation}
Now, involving the variation-of-constants formula
for $v$, we have
\begin{equation}
v(t)=e^{-(A+1)}v(s_0) +\int_{s_0}^{t}e^{-(t-s)(A+1)}u(s) ds,~~ t\in(s_0, T_{max}).
\label{fghbnmcz2.5ghju48cfg924ghyuji}
\end{equation}
Hence, it follows from \dref{eqx45xx12112} and  
 \dref{fghbnmcz2.5ghju48cfg924ghyuji} that
\begin{equation}
\begin{array}{rl}
&\disp{\|(A+1)^\alpha v(t)\|_{L^q(\Omega)}}\\
\leq&\disp{C_{10}\int_{s_0}^{t}(t-s)^{-\alpha-\frac{2}{2}(\frac{1}{p_0}-\frac{1}{q})}e^{-\mu(t-s)}\|u(s)\|_{L^{p_0}(\Omega)}ds+
C_{10}s_0^{-\alpha-\frac{2}{2}(1-\frac{1}{q})}\|v(s_0,t)\|_{L^1(\Omega)}}\\
\leq&\disp{C_{10}\int_{0}^{+\infty}\sigma^{-\alpha-\frac{2}{2}(\frac{1}{p_0}-\frac{1}{q})}e^{-\mu\sigma}d\sigma
+C_{10}s_0^{-\alpha-\frac{2}{2}(1-\frac{1}{q})}\beta.}\\
\end{array}
\label{gnhmkfghbnmcz2.5ghju48cfg924ghyuji}
\end{equation}
Hence, in light of Lemma \ref{lemmaggbb41ffgg},  due to \dref{fghgbhnjcz2.5ghju48cfg924ghyuji}  and \dref{gnhmkfghbnmcz2.5ghju48cfg924ghyuji}, we have
\begin{equation}
\int_{\Omega}|\nabla {v}(t)|^{q}\leq C_{11}~~\mbox{for all}~~ t\in(s_0, T_{max})
\label{ffgbbcz2.5ghju48cfg924ghyuji}
\end{equation}
and $q\in[1,\frac{2{p_0}}{(2-{p_0})^+})$.
Finally, in view of \dref{eqx45xx12112} and \dref{ffgbbcz2.5ghju48cfg924ghyuji},
 we can get \begin{equation}
\int_{\Omega}|\nabla {v}(t)|^{q}\leq C_{12}~~\mbox{for all}~~ t\in(0, T_{max})~~\mbox{and}~~q\in[1,\frac{2{p_0}}{(2-{p_0})^+})
\label{ffgbbcz2.5ghjusseeeddd48cfg924ghyuji}
\end{equation}
with some positive constant $C_{12}.$
Now, due to the Sobolev imbedding theorems and $N=2$, we conclude that
\begin{equation}
\| {v}(\cdot,t)\|_{L^\infty(\Omega)}\leq C_{13}~~\mbox{for all}~~ t\in(0, T_{max}).
\label{ffgbbcz2.5ghjusseeeddd48cfg924gddffhyuji}
\end{equation}

Applying the Young inequality, one obtains from  \dref{1.1sddd6ssdddd3072x}, \dref{ghyyyuuu1.1} and  \dref{ffgbbcz2.5ghjusseeeddd48cfg924gddffhyuji} that for any $p>\max\{2,p_0-1\}$
\begin{equation}
\begin{array}{rl}
&\disp\frac{d}{dt}\disp\int_{\Omega}e^{\xi w}a^p+p(p-1)\int_{\Omega}
e^{\xi w}a^{p-2}|\nabla a|^2+p\mu\int_\Omega e^{2\xi w}a^{p+1}\\
=&\disp{p(p-1)\chi\int_{\Omega}
e^{\xi w}a^{p-1}\nabla a\cdot \nabla v+(p-1)\xi\int_\Omega e^{\xi w}a^{p}vw}\\
&+\disp{\int_\Omega e^{\xi w}a^{p}\{(p+1)+(p-1)\xi\eta w(w-1)+p\mu (1-w)\}}\\
&+\disp{\int_\Omega e^{2\xi w}a^{p+1}(p-1)\xi\eta w}\\
\leq&\disp{\frac{p(p-1)}{2}\int_{\Omega}
e^{\xi w}a^{p-2}|\nabla a|^2+\frac{p(p-1)}{2}\chi^2\int_{\Omega}
e^{\xi w}a^{p}|\nabla v|^2+(p-1)\xi\int_\Omega e^{\xi w}a^{p}vw}\\
&+\disp{\int_\Omega e^{\xi w}a^{p}\{(p+1)+(p-1)\xi\eta w(w-1)+p\mu (1-w)\}}\\
&+\disp{\int_\Omega e^{2\xi w}a^{p+1}(p-1)\xi\eta w}\\
\leq&\disp{\frac{p(p-1)}{2}\int_{\Omega}
e^{\xi w}a^{p-2}|\nabla a|^2+\frac{p(p-1)}{2}\chi^2\int_{\Omega}
e^{\xi w}a^{p}|\nabla v|^2+C_{14}\int_\Omega a^{p+1}}\\
\leq&\disp{\frac{p(p-1)}{2}\int_{\Omega}
e^{\xi w}a^{p-2}|\nabla a|^2+\frac{p(p-1)}{2}\chi^2e^{\xi \rho}\int_{\Omega}
a^{p}|\nabla v|^2+C_{14}\int_\Omega a^{p+1}~~\mbox{for all}~~ t\in(0,T_{max}).}\\
\end{array}
\label{qqqqcz2.511411sssdffrtgysedd4}
\end{equation}
Next, with the help of  the Gagliardo--Nirenberg inequality (see e.g.  \cite{Zhengaass}) yields that
\begin{equation}
\begin{array}{rl}
C_{14}\disp\int_\Omega a^{p+1}=&\disp{
C_{14}\|  {{a^{\frac{p}{2}}}}\|
^{2\frac{(p+1)}{p}}_{L^{2\frac{(p+1)}{p} }(\Omega)}}\\
\leq&\disp{C_{15}(\|\nabla   {{a^{\frac{p}{2}}}}\|_{L^2(\Omega)}^{\mu_1}\|  {{a^{\frac{p}{2}}}}\|_{L^\frac{2p_0}{p}(\Omega)}^{1-\mu_1}+\|  {{a^{\frac{p}{2}}}}\|_{L^\frac{2p_0}{p}(\Omega)})^{2\frac{(p+1)}{p}}}\\
\leq&\disp{C_{16}(\|\nabla   {{a^{\frac{p}{2}}}}\|_{L^2(\Omega)}^{2\mu_1}+1)}\\
=&\disp{C_{16}(\|\nabla   {{u^{\frac{p}{2}}}}\|_{L^2(\Omega)}^{\frac{2(p-p_0+1)}{p}}+1)}\\
\end{array}
\label{cz2.5630ddffgggghhhddff22222ikopl2sdfg44}
\end{equation}
with some positive constants $C_{15}, C_{16}$ and
$$\mu_1=\frac{\frac{{p}}{p_0}-\frac{p}{p+1}}{\frac{{p}}{p_0}}=
\frac{p+1-p_0}{p+1}\in(0,1).$$
Since, $p_0>1$ yields $p_0<\frac{2{p_0}}{2(2-{p_0})^+}$, in light of the H\"{o}lder inequality and \dref{ffgbbcz2.5ghjusseeeddd48cfg924ghyuji}, we derive
\begin{equation}
\begin{array}{rl}
 \disp\frac{\chi^2p({p}-1)}{2}e^{\xi \rho}\disp\int_\Omega{{a^{p }}} |\nabla {v}|^2\leq&\disp{ \disp\frac{\chi^2p({p}-1)}{2}e^{\xi \rho}\left(\disp\int_\Omega{{a^{\frac{p_0}{p_0-1} p }}}\right)^{\frac{p_0-1}{p_0}}\left(\disp\int_\Omega |\nabla {v}|^{2p_0}\right)^{\frac{1}{p_0}}}\\
\leq&\disp{C_{17}\|  {{a^{\frac{p}{2}}}}\|^{2}_{L^{2\frac{p_0}{p_0-1} }(\Omega)},}\\
\end{array}
\label{cz2.57151hhkkhhhjukildrfthjjhhhhh}
\end{equation}
where $C_{17}$ is a positive constant.
Since 
$q_0> 1$ and $p>q_0-1$,
we have
$$\frac{p_0}{p}\leq\frac{p_0}{p_0-1}<+\infty,$$
which together with the Gagliardo--Nirenberg inequality (see e.g.  \cite{Zhengaass})
 implies that
\begin{equation}
\begin{array}{rl}
C_{17}\|  {{a^{\frac{p}{2}}}}\|
^{2}_{L^{2\frac{p_0}{p_0-1} }(\Omega)}\leq&\disp{C_{18}(\|\nabla   {{a^{\frac{p}{2}}}}\|_{L^2(\Omega)}^{\mu_2}\|  {{a^{\frac{p}{2}}}}\|_{L^\frac{2p_0}{p}(\Omega)}^{1-\mu_2}+\|  {{a^{\frac{p}{2}}}}\|_{L^\frac{2p_0}{p}(\Omega)})^{2}}\\
\leq&\disp{C_{19}(\|\nabla   {{a^{\frac{p}{2}}}}\|_{L^2(\Omega)}^{2\mu_2}+1)}\\
=&\disp{C_{19}(\|\nabla   {{a^{\frac{p}{2}}}}\|_{L^2(\Omega)}^{\frac{2(p-p_0+1)}{p}}+1)}\\
\end{array}
\label{cz2.563022222ddrffghhikopl2sdfg44}
\end{equation}
with some positive constants $C_{18}, C_{19}$ and
$$\mu_2=\frac{\frac{{p}}{p_0}-\frac{p}{\frac{p_0}{p_0-1} p }}{\frac{{p}}{p_0}}\in(0,1).$$

Moreover, an application of the Young inequality shows that
\begin{equation}
\begin{array}{rl}
C_{14}\disp\int_\Omega a^{p+1}+\disp\frac{\chi^2p({p}-1)}{2}e^{\xi \rho}\disp\int_\Omega  a^{{p}}|\nabla v|^2 &\leq\disp{
\frac{p({{p}-1})}{4}\int_{\Omega}a^{{{p}-2}}|\nabla a|^2+C_{20}}\\
&\leq\disp{
\frac{p({{p}-1})}{4}\int_{\Omega}e^{\xi w}a^{{{p}-2}}|\nabla a|^2+C_{20}.}\\
\end{array}
\label{111111cz2aasweeeddfff.5ssedfssddff114114}
\end{equation}
Inserting \dref{111111cz2aasweeeddfff.5ssedfssddff114114} into \dref{qqqqcz2.511411sssdffrtgysedd4}, we conclude that
\begin{equation}
\begin{array}{rl}
&\disp{\frac{d}{dt}\disp\int_{\Omega}e^{\xi w}a^p+\frac{p({{p}-1})}{4}\int_{\Omega}
e^{\xi w}a^{p-2}|\nabla a|^2+p\mu\int_\Omega e^{2\xi w}a^{p+1}\leq C_{21}.}\\
\end{array}
\label{cz2aasweee.5ssedfff114114}
\end{equation}
Therefore, integrating the above inequality  with respect to $t$ yields
\begin{equation}
\begin{array}{rl}
\|a(\cdot, t)\|_{L^{{p}}(\Omega)}\leq C_{22} ~~ \mbox{for all}~~p\geq1~~\mbox{and}~~  t\in(0,T_{max}) \\
\end{array}
\label{cz2.5g556789hhjui78jj90099}
\end{equation}
for some positive constant $C_{22}$.
The proof of Lemma \ref{lemma45630223} is complete.
\end{proof}
\begin{remark}
Since, in this paper, we only assume  that $\mu>0$  which is different from  \cite{PangPang1} (see the hypothesis of Lemma 3.2 to  \cite{PangPang1}), firstly  by using the technical lemma (see Lemma \ref{lemma45630223116}), we   could conclude  the boundedness of $\int_{\Omega}{a^{q_0}}(q_0 > 1),$ then
in light of the variation-of-constants formula and  $L^q$-$L^p$ estimates for the heat semigroup, we may finally derive the boundedness of
$\int_{\Omega}{a^{p}}$ (for any $p> 1$).

%

%
\end{remark}

Our main result on global existence and boundedness thereby becomes a straightforward consequence
of Lemma \ref{lemma70} and Lemma \ref{lemma45630223}.

{\bf The proof of Theorem \ref{theorem3}}~

\begin{proof}
Firstly, in light of \dref{1.1sddd6ssdddd3072x},  due to Lemma \ref{lemma45630223}, we derive that there exist positive constants $p_0>2$ and $C_1$ such that
\begin{equation}
\begin{array}{rl}
\|u(\cdot, t)\|_{L^{{p_0}}(\Omega)}\leq C_{1} ~~ \mbox{for all}~~t\in(0,T_{max}). \\
\end{array}
\label{cz2.5g556789hhjssdddui78jj90099}
\end{equation}
Next,
employing  the standard estimate for Neumann semigroup provides $C_2$ and $C_3 > 0$ such that
\begin{equation}
\begin{array}{rl}
&\disp{\|\nabla v(t)\|_{L^\infty(\Omega)}}\\
\leq&\disp{C_2\int_{s_0}^{t}(t-s)^{-\alpha-\frac{2}{2p_0}}e^{-\mu(t-s)}\|u(s)\|_{L^{q_0}(\Omega)}ds+
C_2s_0^{-\alpha}\|v(s_0,t)\|_{L^\infty(\Omega)}}\\
\leq&\disp{C_2\int_{0}^{+\infty}\sigma^{-\alpha-\frac{2}{2p_0}}e^{-\mu\sigma}d\sigma
+C_2s_0^{-\alpha}\beta}\\
\leq&\disp{C_3~~\mbox{for all}~~ t\in(0, T_{max}).}\\
\end{array}
\label{11111gnhmkfghbnmcz2.5ghju48cfg924ghyujiffggg}
\end{equation}
Applying the Young inequality, in light of \dref{1.1sddd6ssdddd3072x} and the first equation of \dref{ghyyyuuu1.1}, one obtains from \dref{11111gnhmkfghbnmcz2.5ghju48cfg924ghyujiffggg} that for any $p\geq4$
\begin{equation}
\begin{array}{rl}
&\disp\frac{d}{dt}\disp\int_{\Omega}e^{\xi w}a^p+p(p-1)\int_{\Omega}
e^{\xi w}a^{p-2}|\nabla a|^2+\int_{\Omega}e^{\xi w}a^p\\
=&\disp{\xi\int_{\Omega}e^{\xi w}a^p\cdot\{-vw+\eta w(1-ae^{\xi w}-w)\}}\\
&+\disp{p\int_\Omega e^{\xi w}a^{p-1}\cdot\{e^{-\xi w}\nabla\cdot(e^{\xi w}\nabla a)-\chi e^{-\xi w}\nabla\cdot(e^{\xi w}a\nabla v)\}}\\
&+\disp{a\xi vw +a(\mu-\xi\eta w)(1-ae^{\xi w}-w)\}+p\int_{\Omega}e^{\xi w}a^p}\\
\leq&\disp{\frac{p(p-1)}{4}\int_{\Omega}
e^{\xi w}a^{p-2}|\nabla a|^2+p(p-1)\chi^2C_4\int_{\Omega}
e^{\xi w}a^{p}}\\
&+\disp{(p-1)\xi\int_\Omega e^{\xi w}a^{p}vw+\int_\Omega e^{\xi w}a^{p}\{(p+1)+(p-1)\xi\eta w(w-1)+p\mu (1-w)\}}\\
&+\disp{\int_\Omega e^{2\xi w}a^{p+1}[(p-1)\xi\eta w-p\mu]}\\
\leq&\disp{\frac{p(p-1)}{4}\int_{\Omega}
e^{\xi w}a^{p-2}|\nabla a|^2+C_5p^2(\int_{\Omega}
a^{p+1}+1)~~\mbox{for all}~~ t\in(0,T_{max}),}\\
\end{array}
\label{cz2.511411sssdfffrtttfgggsedd4}
\end{equation}
where $C_4$ and $C_5$ are independent of $p$. Here and throughout the proof of Theorem \ref{theorem3}, we shall
denote by $C_i(i\in \mathbb{N})$ several positive constants independent of $p$.
Therefore, \dref{cz2.511411sssdfffrtttfgggsedd4} implies that
\begin{equation}
\begin{array}{rl}
&\disp\frac{d}{dt}\disp{\int_{\Omega}e^{\xi w}a^p+C_6\int_{\Omega}
|\nabla a^{\frac{p}{2}}|^2+\int_{\Omega}e^{\xi w}a^p\leq C_5p^2(\int_{\Omega}
a^{p+1}+1)~~\mbox{for all}~~ t\in(0,T_{max}).}\\
\end{array}
\label{cz2.511411sssdffdffffggfrtttfgggsedd4}
\end{equation}
Next, once more by means of the Gagliardo--Nirenberg inequality, we can  estimate
\begin{equation}
\begin{array}{rl}
C_5p^2
\disp\int_\Omega  a^{p+1}=&\disp{ C_5p^2\|a^{\frac{p}{2}}\|_{L^\frac{2(p+1)}{p}(\Omega)}^{\frac{2(p+1)}{p}} }\\
\leq&\disp{ C_7p^2(\|\nabla a^{\frac{p}{2}}\|_{L^{2}(\Omega)}^{\frac{2(p+1)}{p}\varsigma_1}
\| a^{\frac{p}{2}}\|_{L^1(\Omega)}^{\frac{2(p+1)}{p}\varsigma_1}+ \| a^{\frac{p}{2}}\|_{L^1(\Omega)}^{\frac{2(p+1)}{p}}) }\\
=&\disp{ C_7p^2(\|\nabla a^{\frac{p}{2}}\|_{L^{2}(\Omega)}^{\frac{p+2}{p}}
\| a^{\frac{p}{2}}\|_{L^1(\Omega)}+ \| a^{\frac{p}{2}}\|_{L^1(\Omega)}^{\frac{2(p+1)}{p}})}\\
\leq&\disp{  C_6\|\nabla a^{\frac{p}{2}}\|_{L^{2}(\Omega)}^{2}+C_{8}p^{\frac{4p}{p-2}}
\| a^{\frac{p}{2}}\|_{L^1(\Omega)}^{\frac{2p}{p-2}}+  C_7p^2\| a^{\frac{p}{2}}\|_{L^1(\Omega)}^{\frac{2(p+1)}{p}}}\\
\leq&\disp{  C_6\|\nabla a^{\frac{p}{2}}\|_{L^{2}(\Omega)}^{2}+C_{9}p^{\frac{4p}{p-2}}
\| a^{\frac{p}{2}}\|_{L^1(\Omega)}^{\frac{2p}{p-2}},}\\
\end{array}
\label{cz2aasweee.5ssdessderrrffssdffff114114}
\end{equation}
where
$$0<\varsigma_1=\frac{2-\frac{2p}{2(p+1)}}{1-\frac{2}{2}+2}=\frac{p+2}{2(p+1)}<1.$$
Here we have use the fact that $\frac{4p}{p-2}\geq2.$
Therefore, inserting \dref{cz2aasweee.5ssdessderrrffssdffff114114} into \dref{cz2.511411sssdffdffffggfrtttfgggsedd4}, we derive that
\begin{equation}
\begin{array}{rl}
\disp\frac{d}{dt}\disp\int_{\Omega}e^{\xi w}a^p+\int_{\Omega}e^{\xi w}a^p\leq&\disp{ C_{9}p^{\frac{4p}{p-2}}
\| a^{\frac{p}{2}}\|_{L^1(\Omega)}^{\frac{2p}{p-2}}+C_5p^2}\\
\leq&\disp{ C_{10}p^{\frac{4p}{p-2}}\left(\max\{1,
\| u^{\frac{p}{2}}\|_{L^1(\Omega)}\right)^{\frac{2p}{p-2}}.}\\
\end{array}
\label{zjscz2.5297x9630111rrd67ddfff512}
\end{equation}
Now, choosing $p_i=2^{i+2}$ and letting $M_i =\max\{1,\sup_{t\in(0,T)}\int_{\Omega} a^{\frac{{p_i}}{2}}\}$ for $T\in (0, T_{max})$
and $i= 1, 2,3,\cdots$.
Then  we obtain from \dref{zjscz2.5297x9630111rrd67ddfff512} that 
\begin{equation}
\begin{array}{rl}
&\disp{\frac{d}{dt}\int_{\Omega}e^{\xi w}a^{p_i}+\int_{\Omega}e^{\xi w}a^{p_i}\leq C_{11}{p_i^{\frac{2p_i}{p_i-2}}}
M^{\frac{2p_i}{p_i-2}}_{i-1}(T),}\\
\end{array}
\label{zjscz2.5297x9630111rrd67ddfff512df515}
\end{equation}
which, together with the comparison argument entails that there exists a $\lambda>1$ independent of $i$ such that
\begin{equation}
\begin{array}{rl}
&\disp{M_{i}(T)\leq \max\{\lambda^iM^{\frac{2p_i}{p_i-2}}_{i-1}(T),e^{\xi}|\Omega|\|a_0\|_{L^\infty(\Omega)}^{p_i}\}.}\\
\end{array}
\label{zjscz2.5297x9630111rrd6ssdd7ddfff512df515}
\end{equation}
Here we use the fact that $\kappa_i:=\frac{2p_i}{p_i-2}\leq 4.$
Now, if $\lambda^iM^{\kappa_i}_{i-1}(T)\leq e^{\xi\rho}|\Omega|\|a_0\|_{L^\infty(\Omega)}^{p_i}$ for infinitely many
$i\geq 1$, we get
\begin{equation}\left(\sup_{t\in(0,T)}\int_{\Omega} a^{p_{i-1}}(\cdot,t)\right)^{\frac{1}{p_{i-1}}}\leq \left(\frac{e^{\xi\rho}|\Omega|\|a_0\|_{L^\infty(\Omega)}^{p_i}}{\lambda^i}\right)^{\frac{1}{p_{i-1}\kappa_i}}
\label{zjscz2.5297x9ssdrff63011sdertt1rrd6ssdd7ddfff512df515}
\end{equation}
for such $i$, which entails that
\begin{equation}\sup_{t\in(0,T)}\|a(\cdot,t)\|_{L^\infty(\Omega)}\leq \|a_0\|_{L^\infty(\Omega)}.
\label{zjscz2.5297x963011sdertt1rrd6ssdd7ddfff512df515}
\end{equation}

Otherwise, if $\lambda^iM^{\kappa_i}_{i-1}(T)>e^{\xi}|\Omega|\|a_0\|_{L^\infty(\Omega)}^{p_i}$ for all sufficiently large $i$, then by \dref{zjscz2.5297x9630111rrd6ssdd7ddfff512df515}, we derive that
\begin{equation}
\begin{array}{rl}
&\disp{M_{i}(T)\leq \lambda^iM^{\kappa_i}_{i-1}(T)~~~\mbox{for all sufficiently large}~~~i.}\\
\end{array}
\label{zjscz2.5297x9630111rrd6ssdd7ddssddfffffff512df515}
\end{equation}
 Hence, we may choose $\lambda$ large enough such that
\begin{equation}
\begin{array}{rl}
&\disp{M_{i}(T)\leq \lambda^iM^{\kappa_i}_{i-1}(T)~~~\mbox{for all}~~~i\geq1.}\\
\end{array}
\label{zjscz2.5297x9630111rrd6ssdfffffdd7ddssddsddfffffffff512df515}
\end{equation}
 Therefore, based on
 a straightforward induction (see e.g. Lemma 3.12 of \cite{Tao79477}) we have
\begin{equation}
\begin{array}{rl}
\disp M_{i}(T)\leq&\disp{
\lambda^{i+\sum_{j=2}^i(j-1)\cdot\Pi_{k=j}^i\kappa_k}M_{0}^{\Pi_{k=1}^i\kappa_k}~~~\mbox{for all}~~ i \geq 1.}\\
\end{array}
\label{cz2.56303hhyy890678789ty4tt8890013378}
\end{equation}
where $\kappa_k := 2(1+\varepsilon_k)$ satisfies $\varepsilon_k =\frac{4}{p_k-2}\leq \frac{C_{12}}{2^k}$ for all $k\geq 1$ with some $C_{12}> 0$.
Therefore, due to the fact  that  $\ln(1 + x) \leq x (x\geq 0)$, we derive
\begin{equation}
\begin{array}{rl}
\Pi_{k=j}^i:=&\disp{2^{i+1-j}e^{\Sigma_{k=j}^i\ln(1+\varepsilon_j)}}\\
\leq&\disp{2^{i+1-j}e^{\Sigma_{k=j}^i\varepsilon_j}}\\
\leq&\disp{2^{i+1-j}e^{C_{12}}~~~\mbox{for all}~~ i \geq 1~~~\mbox{and}~~ j\in \{1,\ldots,i\},}\\
\end{array}
\label{cz2.563ddfgg03hhyy890678789ty4tt8890013378}
\end{equation}
which implies that
$$
\begin{array}{rl}
\disp\frac{\sum_{j=2}^i(j-1)\cdot\Pi_{k=j}^i\kappa_k}{2^{i+2}}\leq&\disp{\frac{\sum_{j=2}^i(j-1)2^{i+1-j}e^{C_{12}}}{2^{i+2}}}\\
\leq&\disp{\frac{e^{C_{12}}}{2}\sum_{j=2}^i\frac{(j-1)}{2^{j}}}\\
\leq&\disp{\frac{e^{C_{12}}}{2}(\frac{1}{2}+\frac{1}{2^2})}\\
=&\disp{\frac{3e^{C_{12}}}{8}.}\\
\end{array}
$$
By the definition of $p_i$, we easily deduce from \dref{cz2.56303hhyy890678789ty4tt8890013378} that
\begin{equation}
\begin{array}{rl}
\disp M_{i}^{\frac{1}{p_i}}(T)\leq&\disp{
\lambda^{\frac{i}{2^{i+2}}+\frac{\sum_{j=2}^i(j-1)\cdot\Pi_{k=j}^i\kappa_k}{2^{i+2}}}M_{0}^{\frac{\Pi_{k=1}^i\kappa_k}{2^{i+2}}}}\\
\leq&\disp{
\lambda^{\frac{i}{2^{i+2}}}\lambda^{\frac{3e^{C_{12}}}{8}} M_{0}^{\frac{e^{C_{12}}}{4}}.}\\
\end{array}
\label{cz2.56303hhyy89067878dfrrghhh9ty4tt8890013378}
\end{equation}
which after taking $i\rightarrow\infty$ and  $T\nearrow T_{max}$ readily implies that
\begin{equation}\|a(\cdot,t)\|_{L^\infty(\Omega)}\leq \lambda^{\frac{3e^{C_{12}}}{8}} M_{0}^{\frac{e^{C_{12}}}{4}}~~~\mbox{for all}~~~t\in(0,T_{max}).
\label{zjscz2.5297x96ssddd3011sdertt1rrd6ssdd7ddffsdddddf512df515}
\end{equation}

Employing almost exactly the same arguments as in the proof of Lemmata  3.5--3.6 in \cite{PangPang1} (the minor necessary changes are left as an easy
 exercise to the reader), and taking advantage of \dref{11111gnhmkfghbnmcz2.5ghju48cfg924ghyujiffggg} and
 \dref{zjscz2.5297x96ssddd3011sdertt1rrd6ssdd7ddffsdddddf512df515}, we conclude the estimate for any $T<T_{max},$
 \begin{equation}\|\nabla w(\cdot,t)\|_{L^5(\Omega)}\leq C~~~\mbox{for all}~~~t\in(0,T).
\label{zjscz2.5297x96ssddd3011sdertt1rrd6ssdd7ddffdfgggsdddddf512df515}
\end{equation}
Now, with the above estimate in hand, using \dref{zjscz2.5297x963011sdertt1rrd6ssdd7ddfff512df515} and
\dref{zjscz2.5297x96ssddd3011sdertt1rrd6ssdd7ddffsdddddf512df515}, employing the extendibility criterion
provided by Lemma \ref{lemma70},
we may prove Theorem \ref{theorem3}.
\end{proof}

\begin{remark}
If $\mu>\xi\eta\max\{\|u_0\|_{L^\infty(\Omega)},1\}+\mu^*(\chi^2,\xi)$ (see the proof of Lemma 3.4 to  \cite{PangPang1}), one only need to estimate
$Cp^2\int_{\Omega}
a^{p}$ other than $Cp^2(\int_{\Omega}
a^{p+1}+1),$ which is different from this paper.
\end{remark}

{\bf Acknowledgement}:
This work is partially supported by  the National Natural
Science Foundation of China (No. 11601215),
the
Natural Science Foundation of Shandong Province of China (No. ZR2016AQ17) and the Doctor Start-up Funding of Ludong University (No. LA2016006).

\end{document}